\documentclass[11pt,a4paper]{elsarticle}
\usepackage{theorem,enumerate}
\usepackage{amsmath,latexsym,amssymb,amsfonts}
\usepackage{eucal}
\usepackage{mathrsfs}
\usepackage{array}
\usepackage{footmisc}
\usepackage{graphicx, color}
\usepackage{caption,subcaption}

\newtheorem{thm}{Theorem}[section]
\newtheorem{conj}[thm]{Conjecture}
\newtheorem{cor}[thm]{Corollary}
\newtheorem{prop}[thm]{Proposition}
\theorembodyfont{\rmfamily}

\newtheorem{question}{Question}
\newtheorem{problem}{Problem}
\newtheorem{claim}{Claim}

\begin{document}

\begin{frontmatter}

\title{On list 3-dynamic coloring of near-triangulations\\}


\author[1]{Ruijuan Gu}
\ead{millet90@163.com}

\author[2]{Seog-Jin Kim}
\ead{skim12@konkuk.ac.kr}

\author[3]{Yulai Ma\corref{cor1}}
\ead{ylma92@163.com}

\author[3]{Yongtang Shi}
\ead{shi@nankai.edu.cn}
\cortext[cor1]{Corresponding author}

\address[1]{Sino-European Institute of Aviation Engineering, Civil Aviation University of China, Tianjin 300300, China}
\address[2]{Department of Mathematics Educations, Konkuk University, Republic of Korea}
\address[3]{Center for Combinatorics and LPMC, Nankai University,  Tianjin 300071, China}


\begin{abstract}
An $r$-dynamic $k$-coloring of a graph $G$ is a proper $k$-coloring such that for any vertex $v$, there are at least $\min\{r,\deg_G(v) \}$ distinct colors in $N_G(v)$. The {\em $r$-dynamic chromatic number} $\chi_r^d(G)$ of a graph $G$ is the least $k$ such that there exists an $r$-dynamic $k$-coloring  of $G$.  The {\em list $r$-dynamic chromatic number} of a graph $G$ is denoted by $ch_r^d(G)$.
Loeb et al. \cite{LMRW18} showed that $ch_3^d(G)\leq 10$ for every planar graph $G$, and there is a planar graph $G$ with  $\chi_3^d(G)= 7$. 

In this paper, we study a special class of planar graphs which have better upper bounds of $ch_3^d(G)$.
We prove  that $ch_3^d(G) \leq 6$ if $G$ is a planar graph which is near-triangulation, 
where a near-triangulation is a planar graph whose bounded faces are all 3-cycles.
\end{abstract}


\end{frontmatter}
\section{Introduction}

\baselineskip 17.5pt

Let $k$ be a positive integer.
A proper $k$-coloring $\phi: V(G) \rightarrow \{1, 2, \ldots, k \}$ of a graph $G$ is an assignment of colors to the vertices of $G$ so that any two adjacent vertices  receive distinct colors.
The {\em chromatic number} $\chi(G)$ of a graph $G$ is the least $k$ such that there exists a proper $k$-coloring of $G$.
An $r$-dynamic $k$-coloring of a graph $G$ is a proper $k$-coloring $\phi$  such that  for each vertex $v\in V(G)$,
either the number of distinct colors in its neighborhood is at least $r$ or the colors in its neighborhood are all distinct, that is, $|\phi(N_{G}(v))|=\min\{r,\deg_G(v)\}$.
The {\em $r$-dynamic chromatic number} $\chi_r^d(G)$ of a graph $G$ is the least $k$ such that there exists an $r$-dynamic $k$-coloring  of $G$.


A {\em list assignment} on a graph $G$ is a function
$L$ that assigns each vertex $v$ a set $L(v)$ which is
a list of available colors at $v$.
For a list assignment $L$ of a graph $G$, we say $G$ is {\em $L$-colorable}
if there exists a proper coloring $\phi$ such that $\phi(v) \in L(v)$ for every $v \in V(G)$.
A graph $G$ is said to be {\em $k$-choosable} if for any list assignment $L$ such that
$|L(v)| \geq k$ for every vertex $v$, $G$ is $L$-colorable.

For a list assignment $L$ of $G$, we say that $G$ is {\em $r$-dynamically $L$-colorable}
if there exists an  $r$-dynamic coloring  $\phi$ such that $\phi(v) \in L(v)$ for every $v \in V(G)$.
A graph $G$ is {\em  $r$-dynamically $k$-choosable} if for any list assignment $L$ with
$|L(v)| \geq k$ for every vertex $v$, $G$ is $r$-dynamically $L$-colorable.
The  {\em list $r$-dynamic chromatic number}  $ch_r^d(G)$ of a graph $G$ is the least $k$ such that $G$ is  $r$-dynamically $k$-choosable.

An interesting property of dynamic coloring is as follows.
$$\chi(G)\leq\chi_2^d(G)\leq\cdots\leq\chi_{\Delta}^d(G)=\chi(G^2),$$
where $G^2$ is the square of the graph $G$.

The dynamic coloring was first introduced in \cite{LMP03,M01}.  On the other hand, Wegner \cite{W77} conjectured that if $G$ is a planar graph, then

$$\chi_\Delta^d(G)\leq\left\{
\begin{array}{lcl}
7, && {if\ \Delta(G)=3;}\\
\Delta(G)+5, &&{if\ 4\leq\Delta(G)\leq7;}\\
\lfloor\frac{3\Delta(G)}{2}\rfloor+1, &&{if\ \Delta(G)\geq8.}
\end{array}\right.$$

Lai et al. \cite{SFCSL14} posed a similar conjecture about dynamic coloring of planar graphs as follows.

\begin{conj}\label{con1}
Let $G$ be planar graph. Then

$$\chi_r^d(G)\leq\left\{
\begin{array}{lcl}
r+3, && {if\ 1\leq r\leq2;}\\
r+5, &&{if\ 3\leq r\leq7;}\\
\lfloor\frac{3r}{2}\rfloor+1, &&{if\ r\geq8.}
\end{array}\right.$$
\end{conj}

Lai et al. \cite{SLW16} showed the conjecture \ref{con1} is true for planar graphs with girth at least $6$. Several results about $r$-dynamic coloring were published in \cite{AKKNO18,CLLLTZ18,JJOW16,KMW15,KLP13,LMRW18}. In term of maximum average degree, there is a result published in \cite{KP18}. And some special classes of graphs are also investigated, such as $K_4$-minor free graphs \cite{SFCSL14}, $K_{1,3}$-free graphs \cite{LL17} and bipartite graphs \cite{E10}.

Loeb et al. \cite{LMRW18} showed $ch_3^d(G)\leq 10$ if $G$ is a planar graph. 
On the other hand, there is a planar graph $F$ such that $\chi_3^d(F) = 7$.  Thus the following is an interesting problem to study.

\begin{problem} \label{P-dynamic}
What is $\chi_3^d(G)$ if $G$ is a planar graph?  And what is 
$ch_3^d(G)$ if $G$ is a planar graph?
\end{problem}
Currently, we have the following bounds.
\begin{equation} \label{dynamic-gap}
7 \leq \mbox{ max } \{\chi_3^d(G) : G \mbox{ is a planar graph }\} \leq 10.
\end{equation}
It is natural to consider a special class of planar graphs for Problem \ref{P-dynamic}.
Recently, Asayama et al. \cite{AKKNO18} showed that $\chi_3^d(G)\leq 5$ if $G$ is a triangulated planar graph, and the upper bound is sharp.  But, 
we do not know yet whether $ch_3^d(G) \leq 5$ or not, if $G$ is a triangulated planar graph.  

Since there is a gap (\ref{dynamic-gap}) for general case of planar graphs, it  would be interesting to study list  $3$-dynamic chromatic number 
$ch_3^d(G)$ for a special class of planar graphs.  
In this paper, we consider a near-triangulation where
a {\em near-triangulation} is a planar graph whose bounded faces are all 3-cycles. Note that a triangulated planar graph is a special case of a near-triangulation.
First, we show the following theorem.

\begin{thm}\label{NT}
If $G$ is a near-triangulation, then $ch_3^d(G)\leq 6$.
\end{thm}

And  we obtain the following corollary.

\begin{cor} \label{T-main}
If $G$ is a triangulated planar graph, then $ch_3^d(G)\leq 6$.
\end{cor}

Let $W_n$ be the wheel with $n+1$ vertices such that $W_n$ is obtained from an $n$-cycle by adding a new vertex $u$ and joining $u$ and every vertex on the $n$-cycle.
The following can be easily checked.
\begin{prop}\label{prop1}
$ch_3^d(W_n)\leq 6$ for every positive integer $n\geq3$ and $ch_3^d(W_5)=6$.
\end{prop}

Note that Proposition \ref{prop1} and Theorem \ref{NT} imply that the upper bound of list 3-dynamic chromatic number of near triangulations is tight.
And Corollary \ref{T-main} implies that
\[
5 \leq \mbox{ max } \{ch_3^d(G) : G \mbox{ is a triangulated planar graph }\} \leq 6.
\]

Thus it would be interesting to answer the following question.

\begin{question}
Is it true that $ch_3^d(G) \leq 5$ if $G$ is a triangulated planar graph? 
\end{question}

\bigskip
\section{Proof of Theorem \ref{NT}}
Suppose that Theorem \ref{NT} does not hold, and let
$G$ be a minimal counterexample to Theorem \ref{NT}.
Let $C: v_1v_2\cdots v_t v_1$ in counter-clockwise order be the boundary of the outer face of a plane graph $G$. If $|V(G)|\leq6$, then it is easy to obtain $ch_3^d(G)\leq 6$, a contradiction. Hence we have $|V(G)|\geq7$.

First, we prove the following Claim.

\begin{claim}\label{cla1-6} 
For any $v \in V(C)$, we have that  $d_G(v) \geq 4$.
\end{claim}
\begin{proof}
Suppose that there is a vertex $v_k \in V(C)$ with $d_G(v_k)\leq3$. Let $u_1,u_2,\ldots, u_s $ denote the neighbors of $v_k$ in counter-clockwise order.
And let $u_1=v_{k+1}$  ($k+1$ are computed by module $t$).

If $d_G(v_k)=2$ or $d_G(v_k)=3$ with $u_1u_3\in E(G)$, then we remove $v_k$ from $G$ and call the resulting graph by $G'$.  If $d_G(v_k)=3$ and $u_1u_3\notin E(G)$, then we remove $v_k$ from $G$ and add the edge $u_1u_3$ in the outer face, and call the resulting graph by $G'$. 

Let $L'(v) = L(v)$ for every $v \in V(G')$.  
Since $G$ is a minimal counterexample, $G'$ has a $3$-dynamically $L'$-coloring $\phi$.  

If $d_G(v_k)=2$, then there exists a vertex $u_0$ such that $u_0\in (N_G(u_1)\cap N_G(u_2))\setminus\{v_k\} $ since $|V(G)|\geq 7$.
Then we color $v_k$ by a color $c \in L(v_k) \setminus \{\phi(u_0), \phi(u_1),\phi(u_2)\}$, and we obtain that
$G$ has a 3-dynamic coloring from the list assignment $L$, a contradiction.

If $d_G(v_k)=3$, then the vertices $u_1$, $u_2$ and $u_3$ receive distinct colors under the coloring $\phi$. Suppose $u_1u_3\in E(G)$. Then we color $v_k$ by a color $c \in L(v_k) \setminus \{\phi(u_1), \phi(u_2),\phi(u_3)\}$, and we obtain that $G$ has a 3-dynamic coloring from the list assignment $L$. This is a contradiction. Hence suppose that $u_1u_3\notin E(G)$.  If there is a vertex $u_i$ for $i\in\{1, 3\}$ such that $\phi(N_G(u_i))$  has at most two different colors, then we must color $v_k$ by a color $c \in L(v_k) \setminus (\phi(N_G(v_k))\cup \phi(N_G(u_i)))$ so that vertex $u_i$ satisfies the conditions of 3-dynamic coloring. Then one can easily check that the number of forbidden colors at $v_k$ is at most 5 as follows.

Let $S$ be the set consisting of the forbidden colors at $v_k$.
If $|\phi(N_G(u_1))|\geq3$ and $|\phi(N_G(u_3))|\geq3$, then  $S=\{\phi(u_1), \phi(u_2),\phi(u_3)\}$. If $|\phi(N_G(u_i))|\leq2$ and $|\phi(N_G(u_j))|\geq3$ for $\{i,j\}=\{1,3\}$, then  $S=\{\phi(u_1), \phi(u_2),\phi(u_3)\}\cup \phi(N_G(u_i))$ . If $|\phi(N_G(u_1))|\leq2$ and $|\phi(N_G(u_3))|\leq2$, then $S=\{\phi(u_1), \phi(u_2),\phi(u_3)\}\cup \phi(N_G(u_1))\cup\phi(N_G(u_3))$. Since $u_2\in N_G(u_1)\cap N_G(u_3)$, we can easily obtain $|S|\leq5$ for all cases above.

Thus  we can color $v_k$ by a color $c \in L(v_k)$ so that $G$ has a 3-dynamic coloring from the list assignment $L$, and it implies that $G$ is  3-dynamically $L$-colorable.  This is a contradiction, which completes the proof of Claim \ref{cla1-6}.
\qed
\end{proof}

\bigskip
Next, we prove the following Claim.

\begin{claim}\label{cla2-6} 
For any $w\in V(G)\setminus V(C)$, we have that  $d_G(w) \geq 6$.
\end{claim}
\begin{proof}
Suppose that there is a vertex $w$ with $d_G(w) \leq 5$.
Let $w_1,w_2,\ldots,w_s $ denote the neighbors of $w$ in counter-clockwise order.

Suppose $d_G(w)=3$.  We remove $w$ from $G$ and call the resulting graph by $G'$. 
Let $L'(v) = L(v)$ for every $v \in V(G')$. 
Since $G$ is a minimal counterexample, $G'$ has a $3$-dynamically $L'$-coloring $\phi$.   So, we can color $w$ by a color $c \in L(w)\setminus\phi(N_G(w))$ so that $G$ has a 3-dynamic coloring from the list assignment $L$ since $|L(v)| \geq 6$ for each $v \in V(G)$, a contradiction.

Now we suppose $4\leq d_G(w)\leq 5$. With Claim \ref{cla1-6}, we suppose $d_G(v)\geq 4$ for each $v \in V(G)$.
Then we remove $w$ from $G$ and add edges in the face formed by $\{w_1,w_2,\ldots,w_s \}$ so that the resulting graph, denoted  by $G'$, is a near-triangulation. 
Let $L'(v) = L(v)$ for every $v \in V(G')$.  
Since $G$ is a minimal counterexample, $G'$ has a $3$-dynamically $L'$-coloring $\phi$.   
If $N_G(w)  = \{w_1, w_2, \ldots, w_s\}$ has all different colors in the coloring $\phi$, then we color $w$ by a color $c \in L(w) \setminus \{\phi(w_i) : 1 \leq i \leq s\}$.  Then this gives a 3-dynamic coloring from its list assignment $L$.

Next, we consider the case when $N_G(w)  = \{w_1, w_2, \ldots, w_s\}$ have a commen color.  If $w_{i-1}$ and $w_{i+1}$ have a common color in the coloring $\phi$ for some $i \in [s]$ ($i-1$ and $i+1$ are computed by module $s$),  then we color $w$ by a color $c \in L(w) \setminus \{\phi(w_i^{'})\}$ for some $w_i^{'} \in N_G(w_i) \setminus \{w,w_{i-1},w_{i+1}\}$  so that vertex $w_i$ satisfies the condition of 3-dynamic coloring.
It means that we select a neighbor $w_i^{'}$ of $w_i$ and forbid $\phi(w_i^{'})$ at $w$.

Since $d_G(w)  = 4$ or  $d_G(w)  = 5$, then there are at most two vertices $w_i$ such that $\phi(w_{i-1}) = \phi(w_{i+1})$ ($i-1$ and $i+1$ are computed by module $s$).  Then the number of forbidden colors at $w$ is at most 5.

So, 
we can color $w$ by a color $c \in L(w)$ so that $G$ has a 3-dynamic coloring from the list assignment $L$ since $|L(v)| \geq 6$ for each $v \in V(G)$.
Thus $G$ is $3$-dynamically $L$-colorable, which is a contradiction since $G$ is a counterexample.  This completes the proof of Claim \ref{cla2-6}. \qed
\end{proof}

\bigskip
Let $k$ be the number of vertices in $V(G) \setminus V(C)$.  Then $n(G) = t + k$ since $|V(C)| = t$.
Now from Claim \ref{cla1-6} and Claim \ref{cla2-6}, we have
\begin{equation} \label{eq-1}
2e(G)  = \sum_{v \in V(G)} d_G(v) = \sum_{v \in V(C)} d_G(v) + \sum_{v \in V(G)\setminus V(C)} d_G(v) \geq 4t + 6k.
\end{equation}

And since $G$ is a near-triangulation, we have 
\begin{equation} \label{eq-2}
e(G) = 3n(G) - 6 - (|V(C)| - 3) = 3n(G) - t - 3 = 2t + 3k - 3.
\end{equation}
So, by (\ref{eq-1}) and (\ref{eq-2})
\[
4t + 6k - 6   = 2e(G) \geq 4t + 6k \ \Longrightarrow -6 \geq 0,
\]
which is a contradiction.  This completes the proof of Theorem \ref{NT}.
\qed

\bigskip

\noindent {\bf Acknowledgments.}\
Y. Ma and Y. Shi are partially supported by National Natural Science Foundation of China,
Natural Science Foundation of Tianjin (No. 17JCQNJC00300), the China-Slovenia bilateral project ``Some topics in modern graph theory" (No.~12-6),
Open Project Foundation of Intelligent Information Processing Key Laboratory of Shanxi Province (No. CICIP2018005),
and the Fundamental Research Funds for the Central Universities, Nankai University (63191516).
S.-J. Kim's work was supported by the National Research Foundation of Korea (NRF) grant funded by the Korea government (MSIT) (NRF-2018R1A2B6003412). 
\\[0.5cm]

\bigskip

\end{document}